\documentclass{article}
\usepackage{graphicx} 
\usepackage{comment}
\usepackage[english]{babel}
\usepackage{physics}
\usepackage{amsmath,amsthm}
\usepackage[all,cmtip]{xy}
\usepackage{amssymb}
\usepackage{geometry}
\usepackage{fancyhdr}
\usepackage{lipsum}
\usepackage{mathtools}
\usepackage{relsize}
\usepackage{physics}
\usepackage{braket}
\usepackage{bm}
\usepackage{esint}
\usepackage{enumitem}
\usepackage{mathrsfs}
\usepackage{xcolor}
\usepackage{bbold}
\usepackage{comment}

\newtheorem{theorem}{Theorem}

\newtheorem{corollary}[theorem]{Corollary}

\newtheorem{definition}[theorem]{Definition}

\newtheorem{lemma}[theorem]{Lemma}

\newtheorem{proposition}[theorem]{Proposition}
\newtheorem{remark}[theorem]{Remark}

\allowdisplaybreaks

\def\bi{\begin{itemize}}
\def\ei{\end{itemize}}
\def\bem{\begin{emuerate}}
\def\eem{\end{enumerate}}
\def\beq{\begin{eqnarray*}}
\def\eeq{\end{eqnarray*}}

\def\Z{\mathbb{Z}} 
\def\R{\mathbb{R}} 
\def\E{\mathbb{E}} 

\def\Ei{\text{Ei}}

\def\1{\mathbb{1}}

\def\div{\text{div}}
\def\grad{\text{grad}}

\title{Moment estimates for the stochastic heat equation on Cartan-Hadamard manifolds}

\author{Fabrice Baudoin\and Hongyi Chen\and Cheng Ouyang}

\begin{document}

\maketitle

\begin{abstract}
    We study the effect of curvature on the Parabolic Anderson model by posing it over a Cartan-Hadamard manifold. We first construct a family of noises white in time and colored in space parameterized by a regularity parameter $\alpha$, which we use to explore regularity requirements for well-posedness. Then, we show that conditions on the heat kernel imply an exponential in time upper bound for the moments of the solution, and a lower bound for sectional curvature imply a corresponding lower bound. These results hold if the noise is strong enough, where the needed strength of the noise is affected by sectional curvature.
\end{abstract}

\tableofcontents

\section{Introduction}

The main objective of this paper is to study the effects of curvature on the parabolic Anderson model, a stochastic PDE formally expressed as
\begin{equation}\label{eq:pam}
\partial_{t} u(t,x) = \Delta u(t,x) + \beta u(t,x) \cdot W(t,x),\quad u(0,x)=u_0(x).
\end{equation}
The solution of \eqref{eq:pam} is expected to be a random field, a stochastic process $\set{u(t,x)}_{t\in\R_+,x\in M}$ for some space $M$. For us, $M$ will be an $n$-dimensional Cartan-Hadamard manifold, a simply connected complete Riemannian manifold with non-positive sectional curvature. We note that non-random field solutions have been constructed for various SPDEs on manifolds \cite{BDFTphi43,Hairer2023RegularitySO,Mayorcas2023SingularSO}, but study of finer properties remains difficult. $\beta>0$ is a parameter known as inverse temperature, which has been known to exhibit interesting behavior  \cite{CSZ17,Dunlap2020AFS,Dunlap2018FluctuationsOT,GHL24,Tao2022GaussianFO}. $\Delta$ is the Laplace-Beltrami operator, defined by $\Delta f=\div(\grad f)$ for smooth functions $f:M\to \R$. Finally, $W$ is a centered space-time Gaussian noise formally satisfying $$\E[W(t,x)W(t',x')]=\delta(t-t')G(x,x')$$ where $\delta$ is Dirac's delta and $G:M\times M\to \R$ is called the correlation function. \\
In $\R^n,\Z^n$, and graphs, the solution of \eqref{eq:pam} is the partition function for polymer models \cite{ACQ,Comets2004DirectedPI,Cosco2020DirectedPO}. The corresponding free energy $h=\log u$ solves the KPZ equation \cite{ACQ,Bertini1997StochasticBA}, which also models growth of many fluctuating surfaces if $n=1$. Exploration of these topics in our setting is left for future work.\\

It is known that if $G(x,x')=\delta_x(x')$, \eqref{eq:pam} is not solvable in $\R^n$ except when $n=1$. In \cite{Dalang99}, the condition $$\int_{\R^n}\frac{\hat{G}(d\xi)}{1+|\xi|^2}<+\infty,$$ where $\hat{G}$ is the Fourier transform of $G(x,x')=G(x-x')$ was given for the well-posedness of \eqref{eq:pam}. This is Dalang's condition, which has been recognized as a regularity requirement for the noise $W$. Motivated by this, \cite{BAUDOIN2023109920} introduced a family of noises $\set{W_\alpha}_{\alpha\geq 0}$, where the regularity parameter $\alpha\geq 0$ is induced by making the correlation function $G_{2\alpha}$ the kernel of the operator $(-\Delta)^{-\alpha}$ on the Heisenberg group $H^n$ equipped with a sub-Riemannian structure. In Euclidean space $\R^n$, Dalang's condition for this family of noises becomes $\alpha>\frac{n-2}{4}$. Since the topological dimension of $H^n$ is $2n+1$, one may initially expect Dalang's condition to be $\alpha>\frac{2n-1}{4}$, but \cite{BAUDOIN2023109920} showed that the correct condition is in fact $\alpha>\frac{n}{2}$. This is not surprising once one realizes that $\frac{n}{2}=\frac{2(n+1)-2}{4}$, where $2(n+1)$ is the Hausdorff dimension of $H^n$ with the sub-Riemannian structure. Noises of this type were also constructed in \cite{baudoin2024parabolic,Chen2023ParabolicAM}, and Section 2 of this paper will follow a similar construction, where the $\R^n$ result $\alpha>\frac{n-2}{4}$ is shown to still hold.\\

We also investigate when the $p$-$th$ moment of the solution $u(t,x)$ grows exponentially in time. In $\Z^n$ and $\R^n$, this has been related to the intermittent behavior of the solution, which is the presence of high peaks in the graph of any realization of the solution (see \cite{Carmona1994ParabolicAP,Bertini1995TheSH,Hu2014StochasticHE,Molchanov1991IdeasIT}). \cite{TINDEL200253} showed that upper exponential in time bounds for the moments hold for compact manifolds, hinting that intermittency is a local property. \cite{Candil2023ParabolicSP,Chen2023ParabolicAM} gave upper and lower bounds for the moments assuming rough $u_0$ in bounded domains and the torus, respectively. The first positive exponential in time upper and lower moment bounds in a non-compact non-Euclidean manifold appears to be \cite{BAUDOIN2023109920} which as noted prior worked on a sub-Riemannian Heisenberg group. Most recently, \cite{baudoin2024parabolic} was able to show that the exponential growth rate of the moments is affected by geometric quantities on very rough domains, such as fractals. In our case, $M$ is a Cartan-Hadamard manifold with negative curvature and our goal is to determine how the negative curvature environment affects the $p$-$th$ moment growth of the solution. For such manifolds the spectrum has a positive bottom, see \cite{McKean}, one therefore expects that at least for small $\beta$'s the $p$-$th$ moment decays exponentially fast. This heuristic is made very precise in our Theorem \ref {pth moment}, where we prove the following: Let $p\geq 2$, then we have for $u_0\in L^\infty(M)\cap L^r(M)$ 
\begin{align}\label{upper bound intro}
\limsup_{t\to+\infty}\frac{\ln\E[u(t,x)^p]}{t}\leq \frac{p}{2}\left(\Theta_{\alpha}(\sqrt{p-1}\beta)-\frac{(n-1)^2}{\max(2,r)}K_1\right),
\end{align}
where $-K_1$ is an upper bound on the sectional curvature of $M$ and $\Theta_{\alpha}$ is a function which is zero for $\beta$ small enough and such that:
\begin{enumerate}[label=(\roman*)]
        \item if $\alpha\in \left(\frac{n-2}{4},\frac{n}{4}\right)$, $\Theta_{\alpha}(\beta)\sim_{\beta\to\infty}(C\beta^2)^{1-2\alpha+\frac{n}{2}}$,
        \item if $\alpha=\frac{n}{4}$, then $\Theta_{\alpha}(\beta)\sim_{\beta\to+\infty}\frac{C\beta^2}{\ln([C\beta]^2)^2}$,\\
        \item if $\alpha>\frac{n}{4}$, then $\Theta_{\alpha}(\beta)\sim_{\beta\to+\infty}C\beta^2$.
    \end{enumerate}
In particular, for $r=2$ and $\beta$ small enough, we therefore get the upper bound
\[
\limsup_{t\to+\infty}\frac{\ln\E[u(t,x)^2]}{t}\le -\frac{(n-1)^2}{2} K_1
\]
which is consistent with the McKean sharp estimate \cite{McKean} which states that the bottom of the $L^2$-spectrum of $-\Delta$ for $M$ satisfying $\sec_M\leq-K_1<0$ is bounded below by $\frac{(n-1)^2}{4} K_1$. In our Theorem \ref{matching large beta effect}, under the additional assumption of a lower bound on the sectional curvature,  we  remarkably obtain a matching lower bound to \eqref{upper bound intro} in the asymptotic regime $\beta \to +\infty$, therefore establishing the sharpness of \eqref{upper bound intro} in the regime $\beta \to +\infty$.

\

This paper is organized as follows. In section 2 we clarify some notation, state a useful elementary result about the heat kernel on Cartan-Hadamard manifolds, and define a family of noises indexed by a regularity parameter $\alpha\geq 0$, which is well defined given some conditions on the heat kernel, which is implied by a negative sectional curvature upper bound on $M$. In section 3, we prove the well-posedness of \eqref{eq:pam} for $\alpha>\frac{n-2}{4}$, which we call Dalang's condition, and show that \eqref{upper bound intro} holds for these solutions. Finally in section 4 we show that a lower bound on sectional curvature gives rise to the aforementioned matching lower bound for moment asymptotics.

\section{Preliminaries}

We first gather some notation we will use in the paper.
\begin{enumerate}
\item $M$ is a $n$-dimensional Cartan-Hadamard manifold, i.e. $M$ is a Riemannian manifold that is complete and simply connected and has everywhere non-positive sectional curvature.
    \item For $t>0,x,y\in M$, $P_t(x,y)$ is the heat kernel of $M$.
    \item $P^K_t(r)$ for $t>0,r\geq 0$ will denote the heat kernel of a Hadamard manifold with constant sectional curvature $-K$.
    \item $\sec_M\leq (\geq)K$ will be a uniform bound on the sectional curvature of $M$.
    \item $B(x,R)$ will denote a geodesic ball of radius $R$ centered at $x$.
    \item For two functions $f,g:\R\to\R$, $f(x)\sim_{x\to a}g(x)$ means $\lim_{x\to a} \frac{f(x)}{g(x)}=C\neq 0$.
    \item We will write $f\leq C g\leq Ch$ where the first and second $C$ could be different and the second $C$ could depend on more parameters than the first. The same goes for $f\geq C g\geq C h$.
\end{enumerate}

\subsection{Heat kernel Estimates on Cartan-Hadamard manifolds}

We recall the following two results:

\begin{theorem}\label{Cheeger-Yau}
    If $M$ is a Cartan-Hadamard manifold and $\sec_M\leq -K_1\leq 0$, then $$P_t(x,y)\leq P^{K_1}_t(d(x,y)).$$
    If $\sec_M\geq -K_2>-\infty$, then $$P_t(x,y)\geq P^{K_2}_t(d(x,y)).$$
\end{theorem}
\begin{proof}
    See \cite{Cheeger1981ALB}.
\end{proof}

\begin{theorem}\label{Davies-Mandouvalous}
    Suppose $t>0,x,y\in \mathbb{H}^{n}_1$, the $n-$dimensional hyperbolic space. The heat kernel $P^1_t(x,y)$ satisfies
    for some $c,C>0$ depending on $n$ $$ch(t,d(x,y))\leq P^1_t(x,y)\leq Ch(t,d(x,y)),$$
    where the Davies-Mandouvalous function $h:\R_+\times \R_+\to \R+$ is defined by $$h(t,z):=t^{-\frac{n}{2}}(1+t+z)^{\frac{n-3}{2}}(1+z)\exp\left(-\frac{z^2}{4t}-\frac{(n-1)^2t}{4}-\frac{(n-1)z}{2}\right).$$
\end{theorem}
\begin{proof}
    See \cite{Davies1988HeatKB}.
\end{proof}

Our main tool for the rest of the paper is the following theorem.

\begin{theorem}\label{thm: Heat Kernel Bound}
    Let $M$ be a Hadamard manifold.
    Suppose $t>0,x,y\in M$. If $\sec_M\leq -K_1<0$, the heat kernel $P_t(x,y)$ satisfies \begin{equation}\label{DMUB}
      P_t(x,y)\leq C h(K_1t,\sqrt{K_1}d(x,y))
    \end{equation}
    where $C>0$ depends on $n$ and $K_1$. Similarly, if $\sec_M\geq-K_2>-\infty$, we have \begin{equation}\label{DMLB}
        P_t(x,y)\geq ch(K_2t,\sqrt{K_2}d(x,y))
    \end{equation}
    where $c>0$ depends on $n$ and $K_2$.
\end{theorem}
\begin{proof}
    Combine Theorem \ref{Cheeger-Yau} and Theorem \ref{Davies-Mandouvalous} with the identity $$\label{heat kernel curvature scaling}
    P^K_t(r)=K^{\frac{n}{2}}P^1_{Kt}(\sqrt{K}r),
    $$
    which can be derived easily from the fact that scaling the metric tensor by $1/\sqrt{K}$ multiplies all sectional curvatures by $K$.
\end{proof}

We first deduce  the following from the above heat kernel estimate.

\begin{proposition}\label{Heat Kernel contraction and spectral gap}
    Let $M$ be a Riemannian manifold, and $P_t$ be the heat kernel at time $t>0$. 
    \begin{enumerate}[label=(\roman*)]
        \item If $u_0\in L^\infty(M)$, then for all $x\in M$ we have \begin{equation}\label{L inf contraction}
            \abs{P_tu_0(x)}\leq \norm{u_0}_{\infty}
        \end{equation}
        The following assume additionally that $M$ is a Cartan-Hadamard manifold with $\sec_M\leq -K_1<0$.
        \item If $u_0\in L^r(M)\cap L^\infty(M)$, $2\leq r<\infty$, then for all $t>0,x\in M$ we have \begin{equation}\label{2 to inf spec gap}
            \abs{P_tu_0(x)}\leq C e^{-\frac{(n-1)^2}{2r}K_1t}\norm{u_0}_\infty.
        \end{equation}
        \item If $u_0\in L^r(M)\cap L^\infty(M)$, $1\leq r\leq 2$, then for all $t>0,x\in M$ we have \begin{equation}\label{1 to 2 spec gap}
            \abs{P_tu_0(x)}\leq C e^{-\frac{(n-1)^2}{4}K_1t}\norm{u_0}_\infty.
        \end{equation}
    \end{enumerate}
\end{proposition}
\begin{proof} Denote by $x$ a point in $M$. For (i), since $\int_M P_t(x, y) dy\leq 1$,
    $$|P_t u_0(x)| = \left|\int_M P_t(x, y) u_0(y) ~dy\right| \leq \norm{P_t(x,\cdot)}_{1} \norm{u_0}_{\infty} \leq \norm{u_0}_{\infty}.$$
    We claim the following holds for $t\geq 1,x\in M$ if $\sec_M\leq -K_1<0$.
    \begin{itemize}
        \item If $u_0\in L^2(M)$, then there exists $C>0$ such that \begin{equation}\label{L2 spec gap t>1}
            \abs{P_tu_0(x)}\leq Ce^{-\frac{(n-1)^2}{4}K_1t}\norm{u_0}_2.
        \end{equation}
        \item If $u_0\in L^1(M)$, then there exists $C>0$ such that\begin{equation}\label{L1 spec gap t>1}
            \abs{P_tu_0(x)}\leq Ce^{-\frac{(n-1)^2}{4}K_1t}\norm{u_0}_1.
        \end{equation}
    \end{itemize}

    For \eqref{L2 spec gap t>1}, we use the upper bound \eqref{DMUB} which yields: \begin{align*}
        |P_t u_0(x)|^2& =\left| \int_M P_t(x,y) u_0(y) dy \right|^2 \le \int_M P_t(x,y)^2  dy  \int_M u_0^2 (y) dy \\
        & \le p_{2t}(x,x) \| u_0 \|_{2}^2 \le  \frac{C}{t^{3/2}}e ^{-\frac{(n-1)^2}{2}K_1t}\| u_0 \|_{2}^2\le Ce ^{-\frac{(n-1)^2}{2}K_1t}\| u_0 \|_{2}^2.
    \end{align*}
    Taking square roots on both sides proves \eqref{L2 spec gap t>1}. A similar argument works for \eqref{L1 spec gap t>1}: since $u_0\in L^1(M)$, $$|P_t u(x)| = \left|\int_M P_t(x, y) u(y) dy\right| \leq \norm{u_0}_{1}\sup_{y\in M}P_t(x,y)\leq \frac{C}{t^{3/4}}e^{-\frac{(n-1)^2}{4}K_1t}\norm{u_0}_{1}\le Ce^{-\frac{(n-1)^2}{4}K_1t}\norm{u_0}_{1}.$$
    Using the Riesz-Thorin interpolation theorem (see for example \cite{Folland1984RealAM}, Theorem 6.27) with \eqref{L inf contraction} and \eqref{L2 spec gap t>1} gives for $u_0\in L^r(M),2<r<+\infty$ \begin{equation}\label{t>1 r>2 spec gap}
        \abs{P_tu_0(x)}\leq Ce^{-\frac{(n-1)^2}{2r}K_1t}\norm{u_0}_{r}
    \end{equation}
    for $u_0\in L^r(M),2\leq r<+\infty$. Similarly, combining \eqref{L1 spec gap t>1} and \eqref{L2 spec gap t>1} gives\begin{equation}\label{t>1 1<r<2 spec gap} \abs{P_tu_0(x)}\leq Ce^{-\frac{(n-1)^2}{4}K_1t}\norm{u_0}_{r}
    \end{equation}
    for every $t>1$, $u_0\in L^r(M),1 \leq r\leq 2$.\\
    From here, we see that the assumption $u_0\in L^r(M),2\leq r<+\infty$ with \eqref{L inf contraction} and \eqref{t>1 r>2 spec gap} gives us \eqref{2 to inf spec gap} with an elementary argument similar to the one made in \cite{baudoin2024parabolic}, Corollary 2.13.     Similarly, \eqref{1 to 2 spec gap} follows from combining \eqref{L inf contraction} with \eqref{t>1 1<r<2 spec gap}. This finishes the proof.
\end{proof}

\subsection{Fractional Laplacian and Colored Noise}

\begin{definition}\label{kernel for frac Laplacian}
    Suppose $M$ is a Cartan-Hadamard manifold satisfying $\sec_M\leq-K_1<0$. 
    We define a family of kernels $\{G_\alpha\}_{\alpha >0 }$ by 
    \begin{align}\label{fractional correlation}
    G_\alpha(x,y)=\frac{1}{\Gamma(\alpha)}\int_0^{+\infty}t^{\alpha-1}P_t(x,y) dt, \quad x,y \in M, x\neq y.
    \end{align}
    \end{definition}
    Thanks to \eqref{DMUB} the integral \eqref{fractional correlation} is indeed convergent for every $\alpha >0$ when $x\neq y$. Our situation  is therefore quite different from the case of Euclidean space and Sub-Riemannian Heisenberg groups, as seen in \cite{BAUDOIN2023109920}. When $\alpha \le n/2$ the integral \eqref{fractional correlation} is divergent for $x=y$. Indeed as will be seen in Lemma \ref{Lower Bound for correlation}, for $\alpha < n/2$, $G_\alpha(x,y)$ has a singularity of order $1/d(x,y)^{n-2\alpha}$  while for $\alpha = n/2$ the singularity has order $ | \ln d(x,y)|$. When $\alpha > n/2$ the integral \eqref{fractional correlation} is also convergent for $x =y$. 
    
    We note that $G_\alpha$  is the kernel of the operator $(-\Delta)^{-\alpha}$, i.e. for any smooth and compactly supported $f\in C_c^\infty(M)$, $$(-\Delta)^{-\alpha} f(x)=\int_M f(y)G_\alpha(x,y)dy.$$
    We also note that $\alpha=0$ corresponds to Dirac's delta on the diagonal and for $\alpha=1$ the correlation function is the Green's function for the Laplace equation.

\begin{definition}[Sobolev Space $\mathcal{W}^{-\alpha}(M)$] \label{sobo_def}
Suppose $M$ is a Cartan-Hadamard manifold satisfying $\sec_M\leq-K_1<0$. For $\alpha>0$, we define the Sobolev space $\mathcal{W}^{-\alpha}(M)$ as the completion of $ C_c^\infty(M)$ with respect to the inner product
\begin{align}\label{Sobo_norm}
\left\langle  f ,g \right\rangle_{\mathcal{W}^{-\alpha}(M)}=\int_M \int_M f(x) g(y) G_{2\alpha}(x,y) dx dy= \int_{M} \left[(-\Delta)^{-\alpha} f\right](x)\left[ (-\Delta)^{-\alpha} g\right](x) dx,
\end{align}
where the kernel $G_{2\alpha}$ is the family in Definition \ref{kernel for frac Laplacian}.
\end{definition}

\begin{definition}\label{def-frac-Gaus-field}
Consider  $\alpha>0$ and the following Hilbert space of space-time functions:
\begin{align}\label{eq-hilb}
\mathcal{H}_\alpha=L^2(\mathbb{R}_+, \mathcal{W}^{-\alpha}(M)),  
\end{align}
    where $\mathcal{W}^{-\alpha}(M)$ is the Sobolev space introduced in Definition \ref{sobo_def}. On a complete probability space $(\Omega, \mathcal{G},\mathbf{P})$ we define a centered Gaussian family $\{W(\phi); \phi\in \mathcal{H}_\alpha\}$, whose covariance is given by 
\begin{align}\label{eq-Gau-cov}
\mathbf E\left[ W(\varphi) W(\psi)\right]
=&
\int_{\R_+}\ \left\langle \varphi (t,\cdot) , \psi (t,\cdot)\right\rangle_{\mathcal{W}^{-\alpha}(M)}  dt \notag\\
=&
\int_{\R_+} \int_{M^2} \varphi (t,x)  \psi (t,y)G_{2\alpha}(x,y)dxdy dt
\, ,
\end{align}
for $\varphi$, $\psi$ in $\mathcal{H}_\alpha$. This family is called the fractional Gaussian fields on $M$.
\end{definition}

\subsection{$L^2$ estimate for the fractional Laplacian of the heat kernel}

We will use the heat kernel bound to obtain upper  estimates for $\norm{-(\Delta)^{-\alpha}P_t(x,\cdot)}_{L^2(M)}$ which plays a central role in our study.

\begin{lemma}\label{analogue of lemma 2.22}
    Suppose $M$ is a Cartan-Hadamard manifold satisfying $\sec_M\leq-K_1<0$. Let $\alpha> 0$. There exists $C=C_{\alpha,K_1,n}>0$ such that for every $t >0$
    \[
        \sup_{x\in M}\norm{(-\Delta)^{-\alpha}p_t(x,\cdot)}_{L^2(M)}\leq \begin{cases}
        Ce^{-\frac{(n-1)^2}{4}K_1t} ( 1_{[0,1/2K_1]}(t)t^{\alpha-\frac{n}{4}} +1_{[1/2K_1,+\infty)}(t)(1+K_1t)^{-3/4} )& \alpha <\frac{n}{4}\\
        Ce^{-\frac{(n-1)^2}{4}K_1t}( 1_{[0,1/2K_1]}(t)|\ln (t)| +1_{[1/2K_1,+\infty)}(t)(1+K_1t)^{-3/4} )& \alpha =\frac{n}{4}\\
        Ce^{-\frac{(n-1)^2}{4}K_1t} ( 1_{[0,1/2K_1]}(t) +1_{[1/2K_1,+\infty)}(t)(1+K_1t)^{-3/4} )& \alpha >\frac{n}{4}.
        \end{cases}
    \] 
\end{lemma}
\begin{proof}

Applying Fubini to change the order of integration, the Minknowski integral inequality, the semigroup property of the heat kernel, and our heat kernel upper bound, we have
\begin{equation}\label{eq:DeltaHK}
\begin{split}
\|(-\Delta)^{-\alpha}P_t(x,\cdot)\|_{L^2(M)}  
 & \le \frac{1}{\Gamma(\alpha)}\int_0^{+\infty} s^{\alpha-1} P_{2(s+t)}(x,x)^{1/2} ds \\
 &\leq C \int_0^{+\infty} s^{\alpha-1} h(0,2K_1(t+s))^{1/2} ds \\
 &\leq C\int_0^{+\infty} s^{\alpha-1} (t+s)^{-n/4}(1+2K_1(t+s))^{\frac{n-3}{4}} \exp \left( -K_1\frac{(n-1)^2}{4}(t+s)\right) ds\\
 &=C e^{-bK_1t}\int_0^{+\infty} s^{\alpha-1}e^{-bK_1s} [(2K_1(t+s))^{-n}(1+2K_1(t+s))^{n-3}]^{\frac{1}{4}}ds,
\end{split}
\end{equation}
with $b=\frac{(n-1)^2}{4}$ and $T=2K_1(t+s)$ and where the second inequality used \eqref{DMUB}. Notice that $T<1\iff s<\frac{1}{2K_1}-t$. Continuing, we have 
\begin{align*}
    &C e^{-bK_1t}\int_0^{+\infty} s^{\alpha-1}e^{-bK_1s} [(1+T)^{-3}T^{-n}(1+T)^n]^{\frac{1}{4}}ds\\
    = &Ce^{-bK_1t}\int_0^{+\infty} s^{\alpha-1}e^{-bK_1s} \left[(1+T)^{-3}\sum_{j=0}^{n}\binom{n}{j}T^{j-n}\right]^{\frac{1}{4}}ds\\
    \leq& C e^{-bK_1t}\int_0^{+\infty} s^{\alpha-1}e^{-bK_1s} [n!(1+T)^{-3}\max(T^{-n},1)]^{\frac{1}{4}}ds\\
    \leq& C e^{-bK_1t}\left(\int_0^{\max(\frac{1}{2K_1}-t),0)} s^{\alpha-1}e^{-bK_1s} T^{-\frac{n}{4}}ds+\int_{\max(\frac{1}{2K_1}-t),0)}^{+\infty} s^{\alpha-1}e^{-bK_1s} (1+T)^{-\frac{3}{4}}ds\right)\\
    =&C e^{-bK_1t}\left( 1_{[0,\frac{1}{2K_1})}(t)\int_0^{+\infty}s^{\alpha-1}e^{-bK_1s} T^{-\frac{n}{4}}ds+1_{[\frac{1}{2K_1},+\infty)}(t)\int_0^{+\infty}s^{\alpha-1}e^{-bK_1s} (1+T)^{-\frac{3}{4}}ds \right).
\end{align*}

For the first integral, we treat the case $\alpha<\frac{n}{4}$; the other cases are similar or easier. Applying the substitution $s=tu$ gives us \begin{align*}
    \int_0^{+\infty}s^{\alpha-1}e^{-bK_1s} T^{-\frac{n}{4}}ds
    \leq Ct^{\alpha-\frac{n}{4}}\int_0^{+\infty}u^{\alpha-1}(1+u)^{-\frac{n}{4}}e^{-bK_1tu}du
    \leq Ct^{\alpha-\frac{n}{4}}.
\end{align*}
For the second integral, applying $(1,\infty)$ H\"older with $(1+T)^{-\frac{3}{4}}=(1+2K_1(t+s))^{-\frac{3}{4}}\leq (1+K_1t)^{-\frac{3}{4}}$ gives \begin{align*}
    \int_0^{+\infty}s^{\alpha-1}e^{-bK_1s} (1+T)^{-\frac{3}{4}}ds\leq (1+K_1t)^{-\frac{3}{4}}\int_0^{+\infty}s^{\alpha-1}e^{-bK_1s}ds\leq C(1+K_1t)^{-\frac{3}{4}}.
\end{align*}
This completes the proof.
\end{proof}

\section{Well-posedness of the stochastic heat equation and upper bound for moments of the solution}

On the Cartan-Hadamard manifold $M$ we consider the stochastic heat equation

\begin{equation}\label{eq:pam2}
\partial_{t} u(t,x) = \Delta u(t,x) + \beta u(t,x) \cdot W(t,x),\quad u(0,x)=u_0(x).
\end{equation}
where $\beta >0$ is a parameter and $W$ is the fractional noise considered in the previous section. We consider  solutions in the mild sense of It\^o-Walsh,  that is, one rewrites our equation of interest  as
\begin{equation}\label{eq:PAM-mild form}
    u(t,x)=J_0(t,x)+\beta I(t,x),
\end{equation}
where $J_0(t,x)=P_tu_0(x)=\int_M P_t(x,y)u(0,y)dy$ is the solution to the homogeneous heat equation and $I(t,x)$ is the It\^o-Walsh the stochastic integral (see \cite{Walsh1986AnIT} for a rigorous definition of this integral) given by
\begin{equation}
    I(t,x)=\int_0^t\int_{M}P_{t-s}(x,y)u(s,y)W(ds,dy).
\end{equation}
A solution to \eqref{eq:PAM-mild form} is then defined as below.
\subsection{Definition of Solution and Chaos Expansion}
\begin{definition}
    \label{def of Ito sol}
A process $u=\{u(t,x); (t,x)\in\mathbb{R}_{+}\times M \}$ is called a random field solution of~\eqref{eq:pam2} in the It\^o-Walsh sense \cite{Walsh1986AnIT} if the following conditions are met:
\begin{enumerate}
\item $u$ is adapted;

\item $u$ is jointly measurable with respect to $\mathcal{B}(\mathbb{R}_{+}\times M )\otimes \mathcal{G}$; 
\item $\mathbf{E}(I(t,x)^2)<\infty$ for all $(t,x)\in\mathbb{R}_{+}\times M$;
\item The function $(t,x)\to I(t,x)$ is continuous in $L^2(\Omega)$;
\item $u$ satisfies \eqref{eq:PAM-mild form} almost surely for all $(t,x)\in\mathbb{R}_{+}\times M$.
\end{enumerate}
\end{definition}

We will use chaos expansions, the setting for which is given below. Let us first denote by $\mathcal H_k$ the space 
\begin{align}\label{def:H_k}
\mathcal{H}_k=\mathrm{Span}\left(\left\{H_k(W_\alpha(\phi));\phi \in \mathcal H, \|\phi\|_{\mathcal H}=1\right\}\right),
\end{align} 
where $W_\alpha$ is the noise from Definition \ref{def-frac-Gaus-field}, Span means closure of the linear span in $L^2(\Omega)$, $H_k$ designates the $k$-th Hermite polynomial, and $\mathcal{H}=\mathcal{H}_\alpha$ is the Hilbert space from \eqref{eq-hilb}. The space $\mathcal{H}_k$ is called the $k$-th Wiener chaos. There exists a linear isometry $I_k$ between $\mathcal H^{\otimes k}$ (with modified norm $\sqrt{k!}\|\cdot\|_{\mathcal H^{\otimes k}}$) and $\mathcal H_k$ given by
\[
I_k\big(\phi^{\otimes k}\big)
=
k! \, H_k(W_\alpha(\phi)), \quad \text{for any } \phi\in \mathcal H \text{ with }\|\phi\|_{\mathcal H}=1.
\]
Let $u=\{u(t,x); t\ge 0, x\in M\}$ 
be a random field such that $\mathbf \E[u(t,x)^2]<\infty$ for all $t\ge 0$ and $x\in M$.
It is established in \cite{Nualart1995TheMC} that $u(t,x)$ has a Wiener chaos expansion of the form
\begin{equation}\label{eq:chaos expansion}
    u(t,x)=\mathbf \E[u(t,x)]+\sum_{k=1}^{\infty} I_k(f_k(\cdot, t,x)),
\end{equation}
where $f_k$'s are symmetric elements of $\mathcal H^{\otimes k}$ uniquely determined by $u$ and where the series converges in $L^2(\Omega)$. When $u$ is the solution to equation \eqref{eq:pam} according to Definition \ref{def of Ito sol}, by a standard iteration procedure (borrowed from \cite{Hu2007StochasticHE,Hu2014StochasticHE}) $u$ admits a chaos expansion as in \eqref{eq:chaos expansion} with $f_k$ given by 
\begin{equation}\label{eq:fk}
f_k(s_1,y_1,\ldots,s_k,y_k, t,x)=\frac{\beta^k}{k!} p_{t-s_{\sigma(k)}}(x, y_{\sigma(k)})\cdots  p_{s_{\sigma(2)}-s_{\sigma(1)}}(y_{\sigma(2)}, y_{\sigma(1)})  p_{s_{\sigma(1)}}u_0(y_{\sigma(1)}),
\end{equation}
where $\sigma$ denotes the permutation of $\{1,2,\cdots, k\}$ such that $0<s_{\sigma(1)}<\cdots <s_{\sigma(k)}<t$. In this setting we then have the following result.
\begin{proposition}\label{uni-SHE}
A process $\mathcal{U}=\{u(t,x); (t,x)\in \mathbb{R}_+\times M\}$ solves equation \eqref{eq:PAM-mild form} in the sense of Definition \ref{def of Ito sol} if and only if for every $(t,x)$ the random variable $u(t,x)$ admits a chaos decomposition \eqref{eq:chaos expansion}-\eqref{eq:fk}, where the family $\{f_k(\cdot, t,x); k\geq1\}$ satisfies 
\begin{equation}\label{finite chaos}
    \sum_{k=1}^{\infty} k! \|f_k(\cdot, t,x)\|_{\mathcal H^{\otimes k}}^2<\infty.
\end{equation}
\end{proposition}
\noindent
The remainder of the section is devoted to proving that \eqref{finite chaos} holds true for all $\alpha>\frac{n-2}{4}$, which will lead to existence and uniqueness for \eqref{eq:PAM-mild form}. 

\subsection{Chaos and Solution Second Moment Upper Bounds}
\begin{lemma}\label{Analogue of lemma 3.9}
    Suppose $u_0\in L^\infty(M)\cap L^r(M)$, $1\leq r\leq +\infty$.
    Let $b=b(r)=\frac{(n-1)^2}{2\max(2,r)}K_1$ (it is 0 if $r=+\infty$). There exist three functions $F_1,F_2,F_3:\R_+\to\R_+$ such that for every $t>0$, $x\in M$, the following holds for $\alpha>\frac{n-2}{4}$ and $\rho >0$,
    \begin{enumerate}[label=(\roman*)]
        \item If $\frac{n-2}{4}<\alpha<\frac{n}{4}$, $$\norm{f_k(\cdot,t,x)}_{\mathcal{H}^{\otimes k}}^2\leq C^k e^{-2bt}\frac{\beta^{2k}}{k!}\norm{u_0}_{\infty}^2e^{\rho t}F_1(\rho)^k.$$
        \item If $\alpha=\frac{n}{4}$, $$\norm{f_k(\cdot,t,x)}_{\mathcal{H}^{\otimes k}}^2\leq C^k e^{-2bt}\frac{\beta^{2k}}{k!}\norm{u_0}_{\infty}^2e^{\rho t}F_2(\rho)^k.$$
        \item If $\alpha>\frac{n}{4}$, $$\norm{f_k(\cdot,t,x)}_{\mathcal{H}^{\otimes k}}^2\leq C^k e^{-2bt}\frac{\beta^{2k}}{k!}\norm{u_0}_{\infty}^2e^{\rho t}F_3(\rho)^k.$$
    \end{enumerate}
    Here $C>0$ depends on $\alpha,K_1,n,$ and $r$. 
    The functions $F_i:\R_+\to \R_+$ are differentiable and satisfy the following properties:
    \begin{enumerate}[label=(\alph*)]
        \item For $i\in \set{1,2,3}$, and $K_1>0$ we have  
        \[ \dv{F_i}{\rho}<0,\quad\sup_{\rho\in\R_+}F_i(\rho)=F_i(0)<\infty,\text{ and }\lim_{\rho\uparrow+\infty}F_i(\rho)=0.\]
        
        Note that this implies $F_i:\R_+\to (0,F_i(0)]$ is invertible.
        \item For $K_1>0$, we have the asymptotics \begin{equation}\label{F inf rho asymp}
        F_i(\rho)\sim_{\rho\uparrow+\infty} \begin{cases}
            \left(\frac{1}{\rho}\right)^{2\alpha-\frac{n}{2}+1}& i=1,\\
                \frac{(\ln \rho)^2}{\rho}& i=2,\\
                \frac{1}{\rho}& i=3.
            \end{cases}
    \end{equation}
    \end{enumerate}
\end{lemma}
\begin{proof} 
    We follow the proof of Lemma 3.9 of \cite{baudoin2024parabolic}.\\
    \noindent{\it Step 1: Reduction to an integral recursion.} By symmetry, in order to get the desired estimates, we only need to evaluate the $L^2$-norm of $f_k(\cdot,t,x)$ on a particular time simplex $[0,t]^k_<:=\{(s_1,\dots, s_k): 0<s_1<\cdots<s_k<t\}$. For $(s_1,\dots, s_k)\in [0,t]^k_<$ we introduce the following notation:
    \begin{align}\label{def:g_k}
        g_k(s,y,t,x)=p_{t-s_k}(x,y_k)\cdots p_{s_2-s_1}(y_2, y_1),
    \end{align}
    where $y=(y_1, \dots, y_k)$ and $s=(s_1, \dots, s_k)$.
    By comparing \eqref{def:g_k} and \eqref{eq:fk} it is clear that on the simplex $[0,t]^k_<$ we have
    \begin{align}\label{f_k H-norm bounds}
        f_k(s,y,t,x)=\frac{\beta^k}{k!}g_k(s,y,t,x)P_{s_1}u_0(y_1).
    \end{align}
    Hence owing to Proposition \ref{Heat Kernel contraction and spectral gap} (i), (ii), or (iii) (for $u_0\in L^\infty(M)\cap L^r(M)$, $r=+\infty$, $r\in [2,\infty)$ or $r\in [1,2]$, respectively) and taking into account all possible orderings of $s_1,\dots,s_k$  we get
    \begin{equation}\label{f0}
        \|f_k(\cdot, t,x)\|_{\mathcal H^{\otimes k}}^2 \le C \frac{\beta^{2k}}{k!} M_k(t,x)\|u_0\|_{\infty}^2, 
    \end{equation}
    where $M_k(t,x)$ is given by 
    \begin{equation}\label{lemma 3.8 ms 2}
        M_k(t, x)=
        \int_{[0,t]^k_<}\int_{M^{2k}}g_k(s,y,t,x)\prod_{i=1}^k G_{2\alpha}(y_i,y_i')g_k(s,y',t,x) 
        e^{-2b s_1} dydy'ds.
    \end{equation}
    We now upper bound $M_k$ defined by \eqref{lemma 3.8 ms 2} thanks to a recursive procedure. To this aim, taking the expression~\eqref{def:g_k} of $g_k$ into account, observe that one can write \eqref{lemma 3.8 ms 2} as  
    \begin{multline}\label{eq-Mk+1}
        e^{2b t}M_{k+1}(t, x)=\int_0^t e^{2b (t-s_{k+1})} 
        \int_{M^{2}}  p_{t-s_{k+1}}(x,y_{k+1})\\
        \cdot G_{2\alpha}^D(y_{k+1},y'_{k+1}) \, p_{t-s_{k+1}}(x,y'_{k+1})\, e^{2b s_{k+1}}M_k (s_{k+1},y_{k+1}) d\mu(y_{k+1})d\mu(y'_{k+1})\,         ds_{k+1},
    \end{multline}
    where we start the recursion with
    $M_0(t,x):=e^{-2b t}$.
    In the identity above, we shall bound the term $e^{2b s_{k+1}}M_k (s_{k+1},y_{k+1})$ by $N_{k}(s_{k+1})$, where we define
    \begin{equation}\label{eq-Nk}
        N_k(t)=e^{2b t} \sup_{x\in M}M_k(t, x).
    \end{equation}
    Hence resorting to Definition \ref{kernel for frac Laplacian}, we end up with the following relation:
    \begin{equation}\label{recursive1}
        N_{k+1}(t)\le \int_0^t \Psi(t-s) N_{k}(s)ds,
        \quad\text{with}\quad
        \Psi(t)=e^{2b t} \sup_{x\in M}\|(-\Delta)^{-\alpha}P_t(x,\cdot)\|_{L^2(M)}^2.
    \end{equation}
    Notice that one can initiate the recursion \eqref{recursive1} by observing that $N_0(t)=1$.\\
    \noindent{\it Step 2: A renewal procedure.} To study the renewal inequality \eqref{recursive1}  and get the estimates we want to prove, we now appeal to an idea we learned from the proof of Theorem 1.1 in \cite{Khoshnevisan2013NonlinearNE} (see also \cite{Foondun2008IntermittenceAN} for a related circle of ideas). Namely for our additional parameter $\rho >0$ let us denote
    \begin{equation}\label{f1}
        \mathcal{N}_k(\rho)=\sup_{t \ge 0} e^{-\rho t}N_k(t).
    \end{equation}
    From \eqref{recursive1} we infer that
    \[
        e^{-\rho t} N_{k+1}(t)\le \int_0^t e^{-\rho (t-s)}\Psi(t-s) e^{-\rho s} N_{k}(s)ds 
        \le 
        \left( \int_0^t e^{-\rho (t-s)}\Psi(t-s) ds \right) \mathcal{N}_k(\rho).
    \]
    Therefore we have
    \begin{equation}\label{f2}
        \mathcal{N}_{k+1}(\rho) \le \hat{\Psi} (\rho) \, \mathcal{N}_k(\rho) ,
        \quad\text{with}\quad
        \hat{\Psi} (\rho)=\int_0^{+\infty} e^{-\rho s} \Psi (s) ds .
    \end{equation}
    Since $\mathcal{N}_{0}(\rho)=1$ we thus conclude by induction that $\mathcal{N}_k(\rho) \le \hat{\Psi} (\rho)^k$.
    Hence going back to the Definition~\eqref{f1}, for every $t \ge 0$ we obtain
    \begin{align}\label{eq-Nk-est}
     N_k(t)\le \hat{\Psi} (\rho)^k  e^{\rho t} .   
    \end{align}
    Having relation \eqref{eq-Nk} in mind, our bound on $M_k(t)$ is now reduced to an estimate on $\hat{\Psi}$ defined by~\eqref{f2}. We now upper bound the function $\hat{\Psi}$ according to the values of $\alpha$.
    
    Plugging in Lemma \ref{analogue of lemma 2.22} with $C=C_{\alpha,n}(K_1)$ gives us $$\Psi(t)\leq \begin{cases}
        C^2 ( 1_{[0,1/2K_1]}(t)t^{2\alpha-\frac{n}{2}} +1_{[1/2K_1,+\infty)}(t)(1+K_1t)^{-3/2} )& \alpha <\frac{n}{4}\\
        C^2( 1_{[0,1/2K_1]}(t)|\ln (t)|^2 +1_{[1/2K_1,+\infty)}(t)(1+K_1t)^{-3/2} )& \alpha =\frac{n}{4}\\
        C^2(1_{[0,1/2K_1]}(t) +1_{[1/2K_1,+\infty)}(t)(1+K_1t)^{-3/2} )& \alpha >\frac{n}{4}.
    \end{cases}$$
    Plugging in the estimates for $\Psi(t)$ above and using the substitution $u=\rho s$ as well as the integral decomposition $\int_0^{+\infty}=\int_0^{\frac{1}{2K_1}}+\int_{\frac{1}{2K_1}}^{+\infty}$, we have 
    $$\hat{\Psi}(\rho)\leq \begin{cases}
        C (I_1(\rho)+I_4(\rho)),&0<\alpha <\frac{n}{4},\\
        C(I_2(\rho)+I_4(\rho)) &\alpha=\frac{n}{4},\\
        C(I_3(\rho)+I_4(\rho)) &\alpha>\frac{n}{4},
    \end{cases}$$
    where \begin{align*}
        &I_1(\rho):=\int_0^{\frac{1}{2K}}s^{2\alpha-\frac{n}{2}}e^{-\rho s}ds, 
        &I_2(\rho):=\int_0^{\frac{1}{2K}}e^{-\rho s}\ln(s)^2ds,\\
        &I_3(\rho):=\int_0^{\frac{1}{2K}}e^{-\rho s}ds, 
        &I_4(\rho):= \int_{\frac{1}{2K}}^{+\infty}(1+Ks)^{-\frac{3}{2}}e^{-\rho s}ds.
    \end{align*}
    An easy calculus exercise gives $$I_3(\rho)=\frac{1-e^{-\frac{\rho}{2K}}}{\rho},$$ and the substitution $u=\rho s$ gives $$I_1(\rho)=\left(\frac{1}{\rho}\right)^{2\alpha-\frac{n}{2}+1}\gamma\left(2\alpha-\frac{n}{2}+1,\frac{\rho}{2K}\right).$$
    We immediately see that $\alpha\in (0,\frac{n-2}{4}]$ implies $2\alpha-\frac{n}{2}\leq -1$. Thus, $I_1$ only converges if $\alpha> \frac{n-2}{4}$.\\
    We then see that estimates (i) to (iii) are satisfied with 
    \begin{align*}
        F_1(\rho,K):=&I_1(\rho)+I_4(\rho),\\
        F_2(\rho,K):=&I_2(\rho)+I_4(\rho),\\
        F_3(\rho,K):=&I_3(\rho)+I_4(\rho).
    \end{align*}
    We now prove the properties of each $F_i$. For the derivative, $\dv{F_i}{\rho}<0$ follows from $\pdv{\rho}e^{-\rho s}<0$, and the $\rho\uparrow+\infty$ limit follows from a dominated convergence argument with $\lim_{\rho\downarrow 0}e^{-\rho s}=0,s>0$. This implies $\sup_{\rho\in \R_+}F_i(\rho)=F_i(0)$.
    The fact that all $I_j,j\in\set{1,2,3,4}$ are convergent for $\rho=0$ if $K>0$ gives all $F_i(0)<+\infty$ for $K>0$.\\ 
    We now move on to property (b). The $F_1$ and $F_3$ large $\rho$ asymptotics are justified by the asymptotic behaviors of the lower incomplete gamma function in the second argument and $e^{-\frac{\rho}{2K_1}},$ in $\rho$, respectively. The $F_2$ asymptotic is slightly harder. We begin with $$I_2(\rho)\leq \int_0^{+\infty} e^{-\rho s}\ln(s)^2ds$$
    and we will compute the right hand side integral, since it is not trivial. First, observe that differentiation under the integral sign gives $$\dv[n]{a}\int_0^{+\infty} e^{-\rho s}s^ads= \int_0^{+\infty} e^{-\rho s} s^a\ln(s)^nds.$$
    We see that the integral we are interested in is the above with $a=0,n=2$. Now we use the formula $$\int_0^{+\infty} e^{-\rho s}s^ads=\frac{\Gamma(a+1)}{\rho^{a+1}},$$
    which can be shown with the substitution $\rho s=u$. It then follows that $$\int_0^{+\infty} e^{-\rho s}\ln(s)^2ds=\dv[2]{a}\eval_{a=0}\frac{\Gamma(a+1)}{\rho^{a+1}}.$$
    Taking the second derivative twice gives $$\dv[2]{a}\eval_{a=0}\frac{\Gamma(a+1)}{\rho^{a+1}}=\frac{\Gamma''(1)-2\Gamma'(1)\ln(\rho)+\Gamma(1)\ln(\rho)^2}{\rho}.$$
    The identities $\Gamma(1)=1,\Gamma'(1)=-\gamma$, and $\Gamma''(1)=\gamma^2+\frac{\pi^2}{6}$,
    where $\gamma$ is the Euler–Mascheroni constant, gives the most exact representation of our integral of interest possible. This justifies the $F_2$ asymptotic and finishes the proof.
\end{proof}
\begin{remark}\label{Dalang's Condition}
    From the representation of $I_1$ above, we see that $\alpha>\frac{n-2}{4}$ is required for the chaos expansion to converge. This is our version of Dalang's condition.
\end{remark}

\begin{theorem}\label{analogue of theorem 3.10}
    Let $M$ be a Cartan-Hadamard Manifold satisfying $\sec_M\leq -K_1<0$. A mild solution to the Cauchy problem of PAM driven by noise $W_\alpha$ exists and is unique if the initial condition $u_0\in L^\infty(M)\cap L^r(M)$, $1\leq r\leq +\infty$ and $\alpha>\frac{n-2}{4}$. Moreover, with the convention $\frac{1}{+\infty}=0$, its second moment satisfies 
    \begin{equation}\label{2nd Lyapnov Exp UB}
        \limsup_{t\to+\infty}\frac{\ln\E[u(t,x)^2]}{t}\leq \Theta_{\alpha}(\beta)-\frac{(n-1)^2}{\max(2,r)}K_1,
    \end{equation}
where $\Theta_{\alpha}(\beta)$ is a non-negative and (not strictly) increasing function satisfying the following properties.
\begin{enumerate}
    \item Define $$i=i(\alpha):=\begin{cases}
        1& \frac{n-2}{4}<\alpha<\frac{n}{4},\\
        2& \alpha=\frac{n}{4},\\
        3& \alpha>\frac{n}{4}.
    \end{cases}$$ Let $F_i:\R_+\to \R_+$ for $i\in\set{1,2,3}$ be the functions defined in Lemma \ref{Analogue of lemma 3.9}. Then, for $\beta<\frac{1}{\sqrt{CF_i(0)}}$,  $\Theta_{\alpha}\equiv0$. Here $C>0$ is the constant taken from Lemma \ref{Analogue of lemma 3.9}.
    \item The following asymptotic relations hold: \begin{enumerate}[label=(\roman*)]
        \item if $\alpha\in \left(\frac{n-2}{4},\frac{n}{4}\right)$, $\Theta_{\alpha}(\beta)\sim_{\beta\to\infty}(C\beta^2)^{1-2\alpha+\frac{n}{2}}$,
        \item if $\alpha=\frac{n}{4}$, then $\Theta_{\alpha}(\beta)\sim_{\beta\to+\infty}\frac{C\beta^2}{\ln([C\beta]^2)^2}$,\\
        \item if $\alpha>\frac{n}{4}$, then $\Theta_{\alpha}(\beta)\sim_{\beta\to+\infty}C\beta^2$.
    \end{enumerate}
\end{enumerate}
\end{theorem}

\begin{proof}
    For notational simplicity, we write $i=i(\alpha)$.
    We follow the same idea as the proof of Theorem 3.10 in \cite{baudoin2024parabolic}, in which well-posedness holds if $$\sum_{k=1}^{+\infty} k!\norm{f_k(\cdot,t,x)}_{\mathcal{H}^{\otimes k}}<+\infty,$$
    which by Lemma \ref{Analogue of lemma 3.9} is true if $\sum_{k=1}^{+\infty}\{C\beta^2 F_i(\rho)\}^k$ converges as a geometric series for some $\rho$.\\
    Also by Lemma \ref{Analogue of lemma 3.9}, this convergence holds for $\rho=0$ if $C\beta^2<\frac{1}{F_i(0)}$ and for $\rho>F_i^{-1}(\frac{1}{C\beta^2})$ if $C\beta^2\geq \frac{1}{F_i(0)}$ (recall that Lemma \ref{Analogue of lemma 3.9} also tells us all $F_i$ are invertible). We thus define for each $i$ a non-negative, continuous, and (not strictly) increasing function $\Theta_{\alpha}(\beta)$ given by $$\Theta_{\alpha}(\beta):=\begin{cases}
        0,& \beta<\frac{1}{CF_i(0)},\\F_i^{-1}(\frac{1}{C\beta^2}),& \beta\geq\frac{1}{CF_i(0)},
    \end{cases}$$
    The convergence condition for $\rho$ can thus be stated as $\rho>\Theta_{\alpha}(\beta)$. This gives us \eqref{eq:PAM-mild form} is well-posed for every $\beta$ and all $\alpha>\frac{n-2}{4}$, since $\Theta_{\alpha}(\beta)<+\infty$ for all $\beta$. We claim that the second order Lyapunov exponent satisfies \eqref{2nd Lyapnov Exp UB} with this $\Theta_{\alpha}$. Before we prove this claim, we note that property 1. of $\Theta_{\alpha}$ is a direct consequence of the definition, and property 2. follows from property (b) of each $F_i$ stated in Lemma \ref{Analogue of lemma 3.9}. \\
    We will now prove \eqref{2nd Lyapnov Exp UB} with $\Theta_{\alpha}$ defined above. It\^{o} isometry tells us the second moment of the solution is given by $$\E[u(t,x)^2]=P_tu_0(x)^2+\sum_{k=1}^{+\infty} k!\norm{f_k(\cdot,t,x)}_{\mathcal{H}^{\otimes k}}.$$
     Applying Lemma \ref{Analogue of lemma 3.9} to the series, and item (i), (ii), or (iii) (for $r=+\infty$, $2\leq r<+\infty$, or $1\leq r\leq 2$, respectively) of proposition \ref{Heat Kernel contraction and spectral gap} to $\abs{P_tu_0(x)}$ gives for every $\rho>\Theta_{\alpha}(\beta)$, $t\geq 1$,
    \begin{align*}
        \E[u(t,x)^2]&\leq \abs{P_tu_0(x)}^2+\sum_{k=1}^{+\infty} \norm{u_0}_\infty^2 e^{(\rho-2b)t}\left(C\beta^2F_i(\rho)\right)^k\\
        &\leq C'\norm{u_0}_\infty^2e^{(\rho-2b) t}\sum_{k=0}^{+\infty}\left(C\beta^2F_i(\rho)\right)^k,
    \end{align*}
    where $b=\frac{(n-1)^2}{2\max(2,r)}K_1$ and $C'>0$ is from Lemma \ref{Heat Kernel contraction and spectral gap}.
    The $\inf$ of all such $\rho$ is $\Theta_{\alpha}(\beta)$, so
    $$\limsup_{t\uparrow+\infty}\frac{\ln \E[u(t,x)^2]}{t}\leq \Theta_{\alpha}(\beta)-2b.$$
    This finishes the proof.
\end{proof}
\subsection{$P$-th Moment Upper Bound of the Solution}
We now use hypercontractivity to get the $p$-th Lyapunov exponent upper bound.
\begin{theorem}\label{pth moment}
    Let $p\geq 2$, $\Theta_{\alpha}(\beta)$ be as before. Then we have for $u_0\in L^\infty(M)\cap L^r(M)$ $$\limsup_{t\to+\infty}\frac{\ln\E[u(t,x)^p]}{t}\leq \frac{p}{2}\left(\Theta_{\alpha}(\sqrt{p-1}\beta)-\frac{(n-1)^2}{\max(2,r)}K_1\right).$$
\end{theorem}

\begin{proof}
    Recall hypercontractivity in each Wiener chaos \cite[Theorem 1.4.1]{Nualart1995TheMC}, which states \begin{align}\label{eq-C}
        \mathbf E \left( I_k(f_k(\cdot, t,x))^p \right)^{1/p} \le (p-1)^{k/2}\mathbf E \left( I_k(f_k(\cdot, t,x))^2 \right)^{1/2}.
    \end{align}
    Applying this to the chaos expansion in conjunction with Lemma \ref{Analogue of lemma 3.9} gives
    $$\E[\abs{u(t,x)}^p]^{\frac{1}{p}}\leq \abs{P_tu_0(x)}+\sum_{k=1}^{+\infty} [C_{\alpha,n}(K_1)\beta^2(p-1)F_i(\rho)]^{\frac{k}{2}}\norm{u_0}_\infty e^{(\frac{\rho}{2}-\frac{(n-1)^2}{2\max(2,r)}K_1)t}.$$
    Then applying proposition \ref{Heat Kernel contraction and spectral gap} to $\abs{P_tu_0}$ the same way as done in the proof of Theorem \ref{analogue of theorem 3.10} finishes the proof.
\end{proof}

\begin{corollary}\label{upper bound phase transition}
    For $1\leq r\leq +\infty$, let $\beta^c>0$ be the largest positive number satisfying $\Theta_{\alpha}(\beta^c)=\frac{(n-1)^2}{\max(2,r)}K$. If $\sec_M\leq -K<0$ and $u$ solves \eqref{eq:PAM-mild form} with $\beta>0$ satisfying $\sqrt{p-1}\beta<\beta^c$ and $u_0\in L^\infty(M)\cap L^r(M)$, then the $p-$th Lyapunov exponent is $0$ if $r=+\infty$, and is negative if $r<+\infty$.
\end{corollary}
\begin{remark}\label{UB assumption}
    We note here that the definition of our noise for all $\alpha>0$ and results in this section only used that $M$ is a Cartan-Hadamard manifold which has the heat kernel satisfies \eqref{DMUB}
    for some $C,K_1>0$. Thus, the assumption $\sec_M\leq -K_1<0$ on $M$ can be replaced with the heat kernel satisfying \eqref{DMUB} for some $C,K_1$. One nontrivial class of manifolds covered by replacing this assumption is known as the asymptotically hyperbolic manifolds, whose heat kernel satisfying \eqref{DMUB} with $K_1=1$ is proven in \cite{Chen2016TheHK}.
\end{remark}

\section{Lower Bound for Moments}

We now turn to the lower bound for the moments. Our main tool is the following Feynman-Kac representation of the moments.
\subsection{Feynman-Kac Formula for Moments and Correlation Lower Bound}
\begin{theorem}\label{FK rep}
    Let $p\geq 2$, $\set{B^{x,j}}_{1\leq j\leq p}$ be Brownian motions starting at $x$ independent of $W$ and each other. The following Feynman-Kac formula for the $p$-th moment of $u(t,x)$ holds if $u_0\in L^\infty(M)$: $$\E[\abs{u(t,x)}^p]=\E_x\left[\prod_{j=1}^p u_0(B^{x,j}_t)\exp\left\{\beta^2 \sum_{1\leq i,k \leq p,i\neq k}\int_0^t G_{2\alpha}(B^{x,i}_s,B^{x,k}_s)ds\right\}\right].$$
    Here $\E_x$ is the expectation with respect to the Brownian motions $(B^{x,j})_{1\leq j\leq p}$.
\end{theorem}
\begin{proof}
   The proof is done by approximation in the same way as \cite[Theorem 3.13]{baudoin2024parabolic}
\end{proof}

\begin{lemma}\label{Lower Bound for correlation}
    Suppose $\sec_M\geq-K_2>-\infty$. Let $x,y\in M$, and let $z=d(x,y)$ denote the distance between $x$ and $y$. 
    The kernel $G_\alpha(x,y)$ satisfies the lower bound $$G_\alpha(x,y)\geq \overline{G}_{\alpha}(z)$$
    where the positive function $\overline{G}_\alpha$ which satisfies $\dv{\overline{G}_\alpha}{z}<0$ and the asymptotics \begin{align*}
        &\overline{G}_{\alpha}(z)\sim_{z\uparrow+\infty} e^{-z^2},& \alpha>0,\\
        &\overline{G}_{\alpha}(z)\sim_{z\downarrow0} z^{2\alpha-n},&\alpha<\frac{n}{2},\\
        &\overline{G}_{\alpha}(z)\sim_{z\downarrow0} \abs{\log z},&\alpha=\frac{n}{2},\\
        &\overline{G}_{\alpha}(z)\sim_{z\downarrow0} 1,&\alpha>\frac{n}{2}.
    \end{align*}
\end{lemma}
\begin{proof}
    Recall that \eqref{DMLB} holds for $K_2>0$ if $\sec_M\geq-K_2>-\infty$. Using this, we have \begin{align*}
        &G_{\alpha}(x,y):=\frac{1}{\Gamma(\alpha)}\int_0^{+\infty}t^{\alpha-1}P_t(x,y) dt\\
        \geq&C\int_0^{+\infty}t^{\alpha-1} (K_2t)^{-n/2}(1+K_2t+\sqrt{K_2}z)^{\frac{n-3}{2}}(1+\sqrt{K_2}z) \exp \left(-\frac{z^2}{4t}-\frac{(n-1)^2K_2t}{4}-\frac{(n-1)\sqrt{K_2}z}{2} \right) dt\\
        \geq&C \int_0^{+\infty}t^{\alpha-1} (K_2t)^{-n/2}(1+K_2t+\sqrt{K_2}z)^{\frac{n-3}{2}}\exp \left(-\frac{z^2}{4t}-\frac{(n-1)^2K_2t}{4}-\frac{(n-1)\sqrt{K_2}z}{2} \right) dt\\
        =&C\int_0^{+\infty}(K_2)^{1-\alpha} s^{\alpha-1}s^{-n/2}(1+s+\sqrt{K_2}z)^{\frac{n-3}{2}}\exp \left(-\frac{K_2z^2}{4s}-\frac{(n-1)^2s}{4}-\frac{(n-1)\sqrt{K_2}z}{2} \right) \frac{ds}{K_2}\\
        \geq&C\int_0^{+\infty} s^{\alpha-1}s^{-n/2}(1+s+\sqrt{K_2}z)^{\frac{n-3}{2}}\exp \left(-\frac{K_2z^2}{4s}-\frac{(n-1)^2s}{4}-\frac{(n-1)\sqrt{K_2}z}{2}\right)ds.
    \end{align*}
    For $\alpha\leq \frac{n}{2}$, we have 
    \begin{align*}
        &\int_0^{+\infty} s^{\alpha-1}s^{-n/2}(1+s+\sqrt{K_2}z)^{\frac{n-3}{2}}\exp \left(-\frac{K_2z^2}{4s}-\frac{(n-1)^2s}{4}-\frac{(n-1)\sqrt{K_2}z}{2}\right)ds\\
        \geq& \int_0^1 s^{\alpha-1-n/2}(1+s+\sqrt{K_2}z)^{\frac{n-3}{2}}\exp \left(-\frac{K_2z^2}{4s}-\frac{(n-1)^2s}{4}-\frac{(n-1)\sqrt{K_2}z}{2}\right)ds\\
        =& \int_0^1 s^{\alpha-1-n/2}(1+s+\sqrt{K_2}z)^{-\frac{3}{2}}\left(\sum_{j=0}^{n}\binom{n}{j}(1+\sqrt{K_2}z)^js^{n-j}\right)^{\frac{1}{2}}\exp \left(-\frac{K_2z^2}{4s}-\frac{(n-1)^2s}{4}-\frac{(n-1)\sqrt{K_2}z}{2}\right)ds\\
        \geq& \frac{e^{-\frac{(n-1)^2}{4}}-\frac{(n-1)\sqrt{K_2}z}{2}}{(2+\sqrt{K_2}z)^{\frac{3}{2}}}\int_0^1 s^{\alpha-1-n/2}e^{-\frac{K_2z^2}{4s}}ds= \frac{e^{-\frac{(n-1)^2}{4}}-\frac{(n-1)\sqrt{K_2}z}{2}}{(2+\sqrt{K_2}z)^{\frac{3}{2}}}\left(\frac{K_2z^2}{4}\right)^{\alpha-\frac{n}{2}}\int_{\frac{K_2z^2}{4}}^{+\infty}r^{\frac{n}{2}-\alpha-1}e^{-r}dr.
    \end{align*}
    If $\alpha<\frac{n}{2}$, the last integral above is equal to the upper incomplete gamma function (see Definition \ref{upper and lower incomplete gamma functions}) $\Gamma\left(\frac{n}{2}-\alpha,\frac{K_2z^2}{4}\right)$. If instead $\alpha=\frac{n}{2}$, we instead have equality to $-\Ei(-\frac{K_2z^2}{4})$, where Ei is also given in Definition \ref{upper and lower incomplete gamma functions}.
    For $\alpha>\frac{n}{2}$, we have \begin{align*}
        &\int_0^{+\infty} s^{\alpha-1}s^{-n/2}(1+s+\sqrt{K_2}z)^{\frac{n-3}{2}}\exp \left(-\frac{K_2z^2}{4s}-\frac{(n-1)^2s}{4}-\frac{(n-1)\sqrt{K_2}z}{2}\right)ds\\
        \geq& \int_1^{+\infty}s^{\alpha-1-n/2}(1+s+\sqrt{K_2}z)^{\frac{n-3}{2}}\exp \left(-\frac{K_2z^2}{4s}-\frac{(n-1)^2s}{4}-\frac{(n-1)\sqrt{K_2}z}{2}\right)ds\\
        \geq& \int_1^{+\infty}s^{\alpha-1}(1+s+\sqrt{K_2}z)^{-\frac{3}{2}}\exp \left(-\frac{K_2z^2}{4s}-\frac{(n-1)^2s}{4}-\frac{(n-1)\sqrt{K_2}z}{2}\right)ds\\
        =& \int_1^{+\infty}s^{\alpha-1} \left(\sum_{j=0}^3 \binom{3}{j}(1+\sqrt{K_2}z)^js^{3-j}\right)^{-\frac{1}{2}}\exp \left(-\frac{K_2z^2}{4s}-\frac{(n-1)^2s}{4}-\frac{(n-1)\sqrt{K_2}z}{2}\right)ds\\
        \geq& \int_1^{+\infty}s^{\alpha-1}(3!(1+\sqrt{K_2}z)s)^{-\frac{3}{2}}\exp \left(-\frac{K_2z^2}{4s}-\frac{(n-1)^2s}{4}-\frac{(n-1)\sqrt{K_2}z}{2}\right)ds\\
        \geq&C \frac{e^{-\frac{K_2z^2}{4}-\frac{(n-1)\sqrt{K_2}z}{2}}}{(1+\sqrt{K_2}z)^{3/2}} \int_1^{+\infty}s^{\alpha-\frac{5}{2}}e^{-\frac{(n-1)^2s}{4}}ds=C\frac{e^{-\frac{K_2z^2}{4}-\frac{(n-1)\sqrt{K_2}z}{2}}}{(1+\sqrt{K_2}z)^{3/2}}.
    \end{align*}
    Combined, these give us
    \begin{align}\label{exact Gbar}
        G_\alpha(x,y)\geq \begin{cases}
        C\frac{e^{-\frac{(n-1)^2}{4}-\frac{(n-1)\sqrt{K_2}z}{2}}}{\Gamma(\alpha)K_2^{\alpha}(2+\sqrt{K_2}z)^{3/2}} \left(\frac{K_2z^2}{4}\right)^{\alpha-\frac{n}{2}}\Gamma\left(\frac{n}{2}-\alpha,\frac{K_2z^2}{4}\right),& \alpha<\frac{n}{2},\\
        C\frac{e^{-\frac{(n-1)^2}{4}-\frac{(n-1)\sqrt{K_2}z}{2}}}{\Gamma(\alpha)K_2^{\alpha}(2+\sqrt{K_2}z)^{3/2}}(-\Ei(-\frac{K_2z^2}{4})),& \alpha=\frac{n}{2},\\
        C\frac{e^{-\frac{z^2}{4}-\frac{(n-1)\sqrt{K_2}z}{2}}}{\Gamma(\alpha)K_2^{\alpha}(1+\sqrt{K_2}z)^{3/2}},& \alpha>\frac{n}{2}.
    \end{cases}
    \end{align}
    In the above, $C>0$ is different for each $\alpha$ and depends on $n,K_2$.
    We define $\overline{G}_\alpha(z)$ as the right-hand side of \eqref{exact Gbar}. $\dv{\overline{G}_\alpha}{z}<0$ is easily seen from the exact representation, and noting that both the upper incomplete function in the second argument and $-Ei(-z)$ is strictly decreasing.\\ 
    We now prove the claimed asymptotics.
    The $z\uparrow+\infty$ asymptotic is apparent for $\alpha>\frac{n}{2}$. For $\alpha<\frac{n}{2}$, it follows from the large $z$ behavior of the upper incomplete Gamma function in the second argument. For $\alpha=\frac{n}{2}$, it follows from the large $z$ asymptotic of $-\Ei(-z)$.\\
    Observe that in the $\alpha>\frac{n}{2}$ case, $\overline{G}_{\alpha}$ is bounded as $z\downarrow 0$, giving us the last small $z$ asymptotic. For $\alpha<\frac{n}{2}$, the upper incomplete gamma function is always bounded, so the blowup at $0$ is of rate $2\alpha-n$. Finally for the $\alpha=\frac{n}{2}$ case, the blowup rate at $0$ is given by $-\Ei(-z)$, which we recall behaves like $\abs{\ln(z)}$ near 0.
    The proof is thus finished.
\end{proof}
\subsection{Moment Lower Bounds and Matching Growth in large $\beta,p$}
\begin{theorem}\label{thm: moment lb}
    Fix $\alpha>\frac{n-2}{4},\beta>0$. Let $p\geq 2$ be an integer. For any $x\in M,R>0$ if $u_0$ is bounded, $u_0\geq \varepsilon>0$ in $B(x,R)$ and $\sec_M\geq -K_2>-\infty$, we have for $0<r\leq R$ \begin{equation}\label{p-mom lower bd}
        \E[\abs{u(t,x)}^p]\geq C \varepsilon^p\exp\left\{\max \left[\left(Q(r),C\right)\right]pt\right\}
    \end{equation}
    where $Q(r):=\beta^2(p-1)\overline{G}_{2\alpha}(r)-\frac{c}{r^2}-C$, $c,C>0$ depend on $\alpha,n,K_2$ and $R$, and $\overline{G}_\alpha(\cdot)$ is the function defined in Lemma \ref{Lower Bound for correlation}. In particular, if $R=+\infty$, we have 
    \begin{equation}\label{no R dependent lyap exp}
        \liminf_{t\uparrow+\infty,x\in M} \frac{\ln \E[\abs{u(t,x)}^p]}{t}\geq p\max\left(\sup_{0<r<+\infty}Q(r),C\right).
    \end{equation}
\end{theorem}
\begin{proof}
    From the Feynman-Kac formula for moments, we have \begin{align*}
        \E[\abs{u(t,x)}^p]&=\E_x\left[\prod_{j=1}^p u_0(B^{x,j}_t)\exp\left\{\beta^2\sum_{1\leq i,k\leq p,i\neq k}\int_0^t G_{2\alpha}(B^{x,i}_s,B^{x,k}_s)ds\right\}\right]\\
        &\geq \E_x\left[\prod_{j=1}^p u_0(B^{x,j}_t)\exp\left\{\beta^2\sum_{1\leq i,k\leq p,i\neq k}\int_0^t G_{2\alpha}(B^{x,i}_s,B^{x,k}_s)ds\right\}\right]\\
        &\geq \E_x\left[\prod_{j=1}^p u_0(B^{x,j}_t)\1_{t<T^j_{x,R}}\exp\left\{\beta^2\sum_{1\leq i,k\leq p,i\neq k}\int_0^t G_{2\alpha}(B^{x,i}_s,B^{x,k}_s)ds\right\}\right]\\
        &\geq \varepsilon^p \exp{\beta^2 p(p-1)t\overline{G}_{2\alpha}(R)}\E_x\left[\prod_{j=1}^p\1_{t<T^j_{x,R}}\right]\\
        &\geq C\varepsilon^p \exp{\beta^2 p(p-1)t\overline{G}_{2\alpha}(R)-\lambda(x,R)pt}.
    \end{align*}
    The second to last line used the fact that $u_0$ restricted to $B(x,R)$ is bounded below by $\varepsilon>0$ and applied Lemma \ref{Lower Bound for correlation} to every $G_{2\alpha}$. Replacing the Brownian motions with $R$ is allowed due to the stopping time and the fact that $\dv{\overline{G}_{2\alpha}}{z}<0$. The last line used the spectral decomposition of the expectation of stopping time with respect to the Dirichlet Laplace-Beltrami operator.\\ 
    The formula for $Q$ emerges after applying Theorem \ref{Dirichlet eigenvalue comparison}. The result holding for all $r<R$ comes from noting that $B(x,r)\subset B(x,R)$ and thus the proof above can be repeated with $r$ replacing $R$.
\end{proof}
\begin{remark}
    The above lower bound is not sharp. In particular, we note that for $u_0\geq \varepsilon$, the fact that $G_\alpha\geq 0$ tells us that a trivial lower bound $$\E[\abs{u(t,x)}^p]\geq \varepsilon^p$$ must be true, so we would recover the fact that the Lyapunov exponent cannot be negative for initial condition bounded away from 0, which cannot be seen with the result of Theorem \ref{thm: moment lb}. This implies that the current methods have room for improvement. 
\end{remark}
\begin{corollary}\label{lower phase transition}
    Suppose $\alpha>\frac{n-2}{4}$, $x\in M$, $R>0$ and $u_0\in L^\infty(M)$  satisfies $u_0\geq \varepsilon>0$ in some $B(x,R)$. If $u(t,x)$ solves \eqref{eq:PAM-mild form} starting from $u_0$, then for any $p\geq 2$, there exists $\beta_c(p)>0$ depending on $K_2$ such that $\beta>\beta_c(p)$ implies \begin{equation}\label{positive lyap exp}
    \lim_{t\to+\infty}\frac{\ln\E[\abs{u(t,x)}^p]}{t}>0.
    \end{equation}
    Similarly for any fixed $\beta>0$, $\exists p_c(\beta)>0$ depending on $K_2$ such that $p>p_c$ implies \eqref{positive lyap exp}.
\end{corollary}
\begin{proof}
    The expression for $Q(R)$ tells us that for any $R>0$, $\exists \beta(R)$ such that $$\beta>\beta(R)\implies Q(R)>0.$$
    We then define $\beta_c:=\inf_{r\leq R}\beta(R)>0$. Note that if $R=+\infty$, this bound can be taken over all $R$.
    The construction of $p_c$ is nearly identical.
\end{proof}
We now capture the effect of large $\beta$ in the lower bound.

\begin{theorem}\label{matching large beta effect}
    Suppose $M$ is a Cartan-Hadamard manifold satisfying $\sec_M\geq -K_2>-\infty$ and $u_0$ is bounded and satisfies $u_0\geq \varepsilon>0$ in some ball $B(x,R)$. The following holds for integers $p\geq 2$ and $u(t,x)$ which solves \eqref{eq:PAM-mild form} starting from $u_0$.
    \begin{enumerate}[label = {(\Alph*)}]
    \item If $\frac{n-2}{4}<\alpha<\frac{n}{4}$, then $$\lim_{\beta\to+\infty}\lim_{t\to+\infty}\frac{1}{\beta^{\frac{2}{4\alpha-n-2}}}\frac{\ln\E[\abs{u(t,x)}^p]}{t}\geq Cp(p-1).$$
    \item If $\alpha=\frac{n}{4}$, then $$\lim_{\beta\to+\infty}\lim_{t\to+\infty}\frac{\ln(\beta)^2}{\beta^2}\frac{\ln\E[\abs{u(t,x)}^p]}{t}\geq Cp(p-1).$$
    \item If $\alpha>\frac{n}{4}$, then $$\lim_{\beta\to+\infty}\lim_{t\to+\infty}\frac{1}{\beta^2}\frac{\ln\E[\abs{u(t,x)}^p]}{t}\geq Cp(p-1).$$
    \end{enumerate}
    The $C>0$ is different in each of the above inequalities and depend on $\alpha,n,K_2$ and $R$.
\end{theorem}
\begin{proof}
    We follow the proof of Theorem 3.16 in \cite{baudoin2024parabolic}.\\
    By Theorem \ref{thm: moment lb}, we have for all $0<r\leq R$ \begin{equation}\label{useful LE lb}\lim_{t\uparrow+\infty} \frac{1}{t}\ln \E [\abs{u(t,x)}^p]\geq p\max \left(\beta^2(p-1)\overline{G}_{2\alpha}(r)-\frac{c}{r^2}-C,-\frac{(n-1)^2K_2}{4}\right).  
    \end{equation}
    Note that this immediately implies statement (C).\\
    For the other two statements, we see that rom the lower bounds \eqref{exact Gbar} we have for $0<r\leq R$\begin{equation}\label{useful Gbar lb}
     \overline{G}_{2\alpha}(r)\geq \begin{cases}
         C_1 r^{4\alpha-n}, & \alpha<\frac{n}{4}\\
         C_2 \abs{\ln r}, & \alpha =\frac{n}{4}.
     \end{cases}   
    \end{equation}
    The $C_1,C_2>0$ depend on $\alpha, K_2$ and $R$. We note that since $C$ uniformly bounded in $r$, it is irrelevant as $\beta\uparrow+\infty$ so we will ignore it for the rest of the proof.\\
    For (A), since $4\alpha-n>-2$, we that there exists some $c_1$ large enough such that $$C_1c_1^{4\alpha-n}-\frac{c_n}{c_1^2}>0 \text{ holds}.$$
    We now select $r=c_1\beta^{\frac{4\alpha-n-2}{2}}$, which gives us $r<R$ for $\frac{n-2}{4}<\alpha<\frac{n}{4}$ and large $\beta$. This allows us to plug our choice of $r$ into \eqref{useful LE lb} to obtain $$\lim_{t\uparrow+\infty} \frac{1}{t}\ln \E [\abs{u(t,x)}^p]\geq p\mathfrak{c}_1 (p-1)\beta^{\frac{4\alpha-n-2}{2}},$$ which completes the proof of (A).\\
    As in the proof of \cite{baudoin2024parabolic}, Theorem 3.16, we note that for (B), the correct choice for $r$ is $\frac{\mathfrak{c}_2}{\beta}$ for some large enough $\mathfrak{c}_2$, which is derived in a similar manner as above.
\end{proof}

\begin{remark}\label{Matching p-th power growth}
    The same proof can be easily adapted to prove that the $p-$th moment growth also matches those one observes from Theorem \ref{pth moment}. More precisely, we have for fixed $\beta>0$, the following holds assuming the assumptions of Theorem \ref{matching large beta effect}. \begin{enumerate}[label = {(\Alph*)}]
    \item If $\frac{n-2}{4}<\alpha<\frac{n}{4}$, then $$\lim_{p\to+\infty}\lim_{t\to+\infty}\frac{1}{p(p-1)^{\frac{1}{4\alpha-n-2}}}\frac{\ln\E[\abs{u(t,x)}^p]}{t}\geq C\beta^2.$$
    \item If $\alpha=\frac{n}{4}$, then $$\lim_{p\to+\infty}\lim_{t\to+\infty}\frac{\ln(p)^2}{p(p-1)}\frac{\ln\E[\abs{u(t,x)}^p]}{t}\geq C\beta^2.$$
    \item If $\alpha>\frac{n}{4}$, then $$\lim_{p\to+\infty}\lim_{t\to+\infty}\frac{1}{p(p-1)}\frac{\ln\E[\abs{u(t,x)}^p]}{t}\geq C\beta^2.$$
    \end{enumerate}
    As before, statement (C) is obvious from the formula for $Q$ in Theorem \ref{thm: moment lb}.\\
    For the other two cases, the key is to choose $r=c_1(p-1)^{4\alpha-n-2}$ for statement (A), and $r=c_2(p-1)^{-\frac{1}{2}}$ statement (B) at the same place where $r$ was chosen in terms of $\beta$. 
\end{remark}

\begin{remark}\label{R dependence}
    The dependence on $R$ for all constants in this section as well as $\beta_c(p)$, $\rho_c(p)$ in Corollary \ref{lower phase transition} was noted to include $u_0$ with compact support. The reader who is only concerned about $u_0\geq \varepsilon>0$ uniformly can read all the theorems here without $R$ dependence thanks to \eqref{no R dependent lyap exp}.
\end{remark}

To conclude the paper, we note that by combining the upper and lower bounds we easily obtain some intermittency property.

\begin{corollary}\label{corollary: Moment Intermittency}
    Fix $\alpha>\frac{n-2}{4},\beta>0$ and some $x\in M$. Suppose that $P_t(x,y)$ satisfies \eqref{DMUB} for some $K_1>0$, $\sec_M\geq-K_2>-\infty$ and $u_0\in L^\infty(M)\cap L^r(M)$ satisfies $\abs{u_0}\geq \varepsilon>0$ in $ B(x,R)$ for some $R>0$. Then for any $q\geq 2$, we have some $p_0>q$ such that if $p>p_0$, $u$ satisfies \begin{equation}\label{mom interm}\lim_{t\uparrow+\infty}\frac{\E[\abs{u(t,x)}^p]^\frac{1}{p}}{\E[\abs{u(t,x)}^q]^\frac{1}{q}}=+\infty.
    \end{equation}
\end{corollary}
\begin{proof}
    We follow the proof of \cite{baudoin2024parabolic}, Corollary 3.21. For $p,q\geq 2$, we let $$\Tilde{R}_{p,q}(t,x):=\frac{\E[\abs{u(t,x)}^p]^\frac{1}{p}}{\E[\abs{u(t,x)}^q]^\frac{1}{q}}.$$
    Applying Theorem \ref{analogue of theorem 3.10} to $\E[\abs{u(t,x)}^q]^\frac{1}{q}$ and Theorem \ref{thm: moment lb} to $\E[\abs{u(t,x)}^p]^\frac{1}{p}$, we have for large enough $p$ $$\frac{1}{t}\ln \Tilde{R}_{p,q}(t,x)\geq \beta^2 (p-1)\overline{G}_{2\alpha}(R)-\frac{c}{R^2}-C-\Theta_{\alpha}(\sqrt{q-1}\beta)+\frac{(n-1)^2K_1}{2\max(2,r)}.$$
    Since everything in RHS is fixed except for $p$, we can pick $p_0$ to be the smallest integer value of $p$ such that the quantity on the right hand side above is non-negative, which concludes the proof.
\end{proof}

\section{Appendix}

\begin{definition}[Some non-elementary functions]\label{upper and lower incomplete gamma functions}
        For $s>0$, $x>0$, $\Gamma(s,x):=\int_x^{+\infty} t^{s-1}e^{-t}dt$ is the upper incomplete gamma function, and $\gamma(s,x):=\int_0^x t^{s-1}e^{-t}dt$ is the lower incomplete gamma function. For $s>0$ they satisfy $\Gamma(s)=\gamma(s,x)+\Gamma(s,x)$, where $\Gamma(s)$ is the well-known gamma function with domain $s>0$. The above additive relation tells us that for $s>0$,$$\sup_{x>0}\gamma(s,x)=\sup_{x>0}\Gamma(s,x)=\Gamma(s).$$
        If $s=0,$ the additive relation fails due to $\gamma(0,x)$ being undefined. However, we have $\Gamma(0,x)=-\Ei(-x)$, where $\Ei(z):=-\int_{-z}^{+\infty} \frac{e^{-t}}{t}dt$ is the exponential integral function. 
        The following asymptotic results follow from elementary calculus arguments.
        \begin{enumerate}
            \item For $s>0$, $\gamma(s,x)\sim_{x\downarrow 0}x^{s}.$
            \item For $s\geq 0$, $\Gamma(s,x)\sim_{x\uparrow+\infty}e^{-x}$.
            \item $-\Ei(-x)\sim_{x\downarrow0}-\ln(x)$.
        \end{enumerate}
\end{definition}

\begin{theorem}[Dirichlet Eigenvalue Estimates]\label{Dirichlet eigenvalue comparison}
    Suppose $M$ is a $n-$dimensional complete Riemannian manifold satisfying $\sec_M\geq -K>-\infty$. Denote by $\lambda(x,R)$ the first eigenvalue of the Dirichlet Laplace-Beltrami operator on the ball of radius $R$ centered at $x$ in $M$. Then, $$\lambda(x,R)\leq \frac{c}{R^2}+C.$$
    Here $c>0$ depending on $n$ and $C>0$ depending on $n,K$ bounded are given in \cite{Berge2023eigenvalues}, Theorem 3.1.
\end{theorem}
\begin{proof}
    By \cite{Chavel1984EigenvaluesIR} Chapter 3, Theorem 7, $$\lambda(x,R)\leq \lambda^K(R),$$ where $\lambda^K(R)$ denotes the first eigenvalue of the Dirichlet Laplace-Beltrami on any ball of radius $R$ in $H^n_K$. By \cite{Berge2023eigenvalues}, Theorem 3.1, $$\lambda^K(R)\leq \frac{c}{R^2}+\frac{(n-1)^2K}{4}+\left(\frac{(n-2)^2}{4}+\frac{1}{4}\right)\left(\frac{1}{\sinh(R\sqrt{-K})/\sqrt{-K}}-\frac{1}{R^2}\right).$$ 
    By Remark 3.2. of \cite{Berge2023eigenvalues}, $\frac{1}{\sinh(R\sqrt{-K})/\sqrt{-K}}-\frac{1}{R^2}$ is bounded in $R$, so $$\frac{(n-1)^2K}{4}+\left(\frac{(n-2)^2}{4}+\frac{1}{4}\right)\left(\frac{1}{\sinh(R\sqrt{-K})/\sqrt{-K}}-\frac{1}{R^2}\right)\leq C,$$
    where $C>0$ depends on $n,K$. This finishes the proof.
\end{proof}


\bibliographystyle{alpha}
\bibliography{citations}

\end{document}